\documentclass{amsart}

\usepackage{enumerate}
\usepackage[latin1]{inputenc}
\usepackage{dsfont}
\usepackage{amssymb,amsthm,amsmath}

\input{xy}
\xyoption{all}

\usepackage{mathrsfs}

\theoremstyle{plain}
\newtheorem{prop}{Proposition}[section]
\newtheorem{lemma}[prop]{Lemma}

\newtheorem{thm}[prop]{Theorem}
\newtheorem*{thmnn}{Theorem}

\theoremstyle{definition}

\theoremstyle{remark}
\newtheorem{remark}[prop]{Remark}

\DeclareMathOperator{\PSL}{PSL}

\DeclareMathOperator{\PGamma}{P\Gamma}

\DeclareMathOperator{\Ima}{Im}
\DeclareMathOperator{\Rea}{Re}

\newcommand{\dec}{\text{dec}}
\newcommand{\parab}{\text{par}}

\newcommand\R{\mathbb{R}}
\newcommand\Z{\mathbb{Z}}
\newcommand\C{\mathbb{C}}

\newcommand{\h}{\mathbb{H}}

\newcommand{\mc}[1]{\mathcal #1}

\newcommand{\wt}{\widetilde}
\newcommand{\wh}{\widehat}

\DeclareMathOperator{\FE}{FE}

\DeclareMathOperator{\id}{id}

\DeclareMathOperator{\Fct}{Fct}
\DeclareMathOperator{\MCF}{MCF}

\newcommand{\sceq}{\mathrel{\mathop:}=}
\newcommand{\seqc}{\mathrel{=\mkern-4.5mu{\mathop:}}}

\newcommand{\bmat}[4]{\begin{bmatrix} #1&#2\\#3&#4\end{bmatrix}}

\newcommand{\textbmat}[4]{\left[\begin{smallmatrix} #1&#2 \\ #3&#4
\end{smallmatrix}\right]}

\usepackage[colorlinks]{hyperref}
\usepackage{graphics}
\usepackage[T1]{fontenc} 

\begin{document}

\title[Period functions for Maass cusp forms for $\Gamma_0(p)$]{Period functions for Maass cusp forms for $\Gamma_0(p)$: \\ a transfer operator approach}
\author[A.\@ Pohl]{Anke D.\@ Pohl}
\address{ETH Z\"urich, Departement Mathematik, R\"amistrasse 101, CH-8092 Z\"urich}
\email{anke.pohl@math.ethz.ch}
\begin{abstract} 
We characterize the Maass cusp forms for Hecke congruence subgroups of prime level as $1$-eigenfunctions of a finite-term transfer operator.
\end{abstract}
\thanks{The author acknowledges the support by the SNF (Grant 200021-127145).}
\maketitle

% flatex input: [intro.tex]
\section{Introduction}

Let $\h$ denote the hyperbolic plane and let $\Gamma$ be a Fuchsian group. Maass cusp forms for $\Gamma$ are specific eigenfunctions of the Laplace-Beltrami operator acting on $L^2(\Gamma\backslash \h)$ which decay rapidly towards any cusp of $\Gamma\backslash \h$. They are of utmost importance in various fields of mathematics. To name but one example, the space $L^2_{\text{discrete}}(\Gamma\backslash\h)$ of the discrete spectrum is spanned by the Maass cusp forms for $\Gamma$ and the constant functions.

For the projective Hecke congruence subgroups
\[
\Gamma_p \sceq \PGamma_0(p) = \left\{ \bmat{a}{b}{c}{d} \in \PSL(2,\Z) \left\vert\ c \equiv 0 \mod p \vphantom{\bmat{a}{b}{c}{d}}\right.\right\}
\]
of any prime level $p$ we provide a dynamical approach to their Maass cusp forms via transfer operator families parametrized in the spectral parameter $s$. Our main theorem is as follows:

\begin{thmnn}
For $s\in\C$, $0<\Rea s<1$, the space of Maass cusp forms for $\Gamma_p$ with eigenvalue $s(1-s)$ is isomorphic (as vector space) to the space of highly regular $1$-eigenfunctions of the transfer operator with parameter $s$. 
\end{thmnn}

The required regularity of these eigenfunctions is specified in Theorem~\ref{mainthm_coarse} below. The transfer operators arise from a discretization of the geodesic flow on the orbifold $\Gamma_p\backslash\h$ which was specifically developed in \cite{Hilgert_Pohl, Pohl_Symdyn2d} for such a dynamical approach to Maass cusp forms. These transfer operators are finite sums of specific elements from $\Gamma_p$ acting via principal series representation on functions defined on certain intervals in the geodesic boundary of $\h$. Therefore their $1$-eigenfunctions are characterized as solutions of finite families of finite-term functional equations. The $1$-eigenfunctions in the main theorem can be understood as period functions for the Maass cusp forms for $\Gamma_p$.

In Section~\ref{sec_background} below we recall the discretization, provide the associated transfer operators and state a definition of period functions.  Employing the recent characterization of Maass cusp forms in parabolic $1$-cohomology by \cite{BLZ_part2} we prove in Section~\ref{proofmainthm} that period functions are in isomorphism with parabolic $1$-cocycle classes by constructing exactly one representative for each such class. 

The results in this article depend on a number of choices. In Section~\ref{sec_examp} we briefly show, for the lattice $\Gamma_3$, the effect of a different choice. A more detailed discussion of these variations as well as a generalization to other nonuniform Fuchsian groups will appear in a forthcoming article. 

Definitions of period functions for these lattices have been provided via other approaches as well, e.g.\@ by \cite{Deitmar_Hilgert} and \cite{Chang_Mayer_eigen} (see also \cite{Hilgert_Mayer_Movasati} and \cite{Fraczek_Mayer_Muehlenbruch}). Clearly, the spaces of these period functions are isomorphic to those provided here. It would be interesting to understand the precise isomorphism. 

% flatex input end: [intro.tex]

% \usepackage[notref,notcite]{showkeys}
% flatex input: [discretization.tex]
\section{Discretization, transfer operators and period functions}\label{sec_background}

The definition of period functions we provide in this article relies on a specific discretization of the geodesic flow on $Y_p\sceq \Gamma_p\backslash\h$, constructed in \cite{Hilgert_Pohl, Pohl_Symdyn2d}. Its construction starts with constructing a cross section (in the sense of Poincar\'e) and choosing a convenient set of representatives with the help of a Ford fundamental domain. We start this section by recalling the steps essential for its present purpose. For proofs and a more detailed exposition we refer to \cite{Hilgert_Pohl, Pohl_Symdyn2d}. Contrary to other applications of cross sections, here we only have to require that all periodic geodesics on $Y_p$ intersect the cross section. After this brief presentation of the construction, we provide the associated transfer operator families and give a definition of period functions.

For all practical purposes we use the upper half plane
\[
 \h = \{ z\in\C \mid \Ima z> 0\}
\]
as model for the hyperbolic plane (even though the construction does not depend on the choice of the model). Further we identify its geodesic boundary with $P^1(\R) \cong \R \cup \{\infty\}$. The action of $\Gamma_p$ on $\h$ is then by M\"obius transformations, hence
\[
 \bmat{a}{b}{c}{d}.z = \frac{az+b}{cz+d}
\]
for all $\textbmat{a}{b}{c}{d}\in\Gamma_p$ and $z\in\h$. This action extends continuously to $P^1(\R)$. We call a point in $P^1(\R)$ cuspidal if it is stabilized by some element in $\Gamma_p$. Finally, a smooth function refers to a $C^\infty$ function.

\subsection{Ford fundamental domain} 
Underlying to all constructions is the Ford fundamental domain 
\[
 \mc F \sceq \left( (0,1) + i\R^+\right) \cap \bigcap_{k=1}^{p-1} \left\{ z\in\h \left\vert\ \left| z-\frac{k}{p}\right| > \frac{1}{p}\right.\right\}.
\]
It is a fundamental domain for $\Gamma_p$ with side pairing given as follows. The left vertical side is mapped to the right vertical side by the element
\begin{equation}\label{T}
 T \sceq \bmat{1}{1}{0}{1}.
\end{equation}
For each $k\in \{1,\ldots, p-1\}$ let 
\[
 h_k \sceq \bmat{k'}{-\frac{kk'+1}{p}}{p}{-k}
\]
be the unique element in $\Gamma_p$ with $k'\in \{1,\ldots, p-1\}$. Then the element $h_k$ maps the boundary arc of $\mc F$ contained in 
\[
 J_k \sceq \left\{ z\in\h \left\vert\ \left|z - \frac{k}{p}\right|=\frac{1}{p} \right.\right\}
\]
bijectively to that contained in $J_{k'}$. 

The function $k\mapsto k'$ depends on the value of the prime level $p$. However, we have
\[
 h_k^{-1} = h_{k'}
\]
for any $k$, and 
\[
 h_1 = \bmat{p-1}{-1}{p}{-1} = h_{p-1}^{-1}
\]
for any $p$. 

We use the side pairing to understand the local structure of the neighboring translates to $\mc F$ and to derive a presentation for $\Gamma_p$. For $k\in\{0,\ldots, p-1\}$ we let $A_k$ denote the strip in $\mc F$ between $\frac{k}{p}$ and $\frac{k+1}{p}$. Further, we set $A_p \sceq T^{-1}.A_{p-1}$ (and discard $A_{p-1}$). The relevant part of the effect of the side pairing elements on $A_0,\ldots, A_p$ is indicated in Figures~\ref{glue1pp} and \ref{glue4pp}. 

\begin{figure}[h]
\begin{center}
\includegraphics{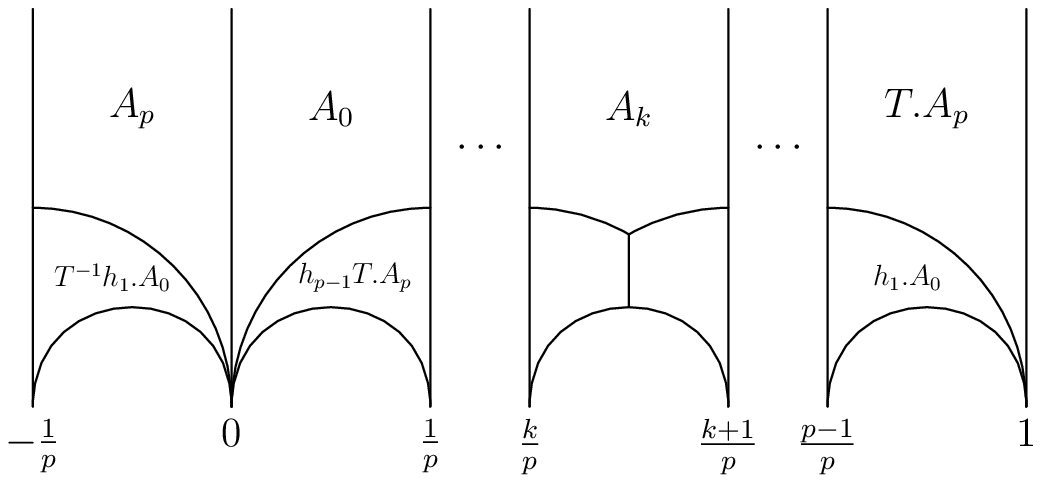}
\end{center}
\caption{}
\label{glue1pp}
\end{figure}

\begin{figure}[h]
\begin{center}
\includegraphics{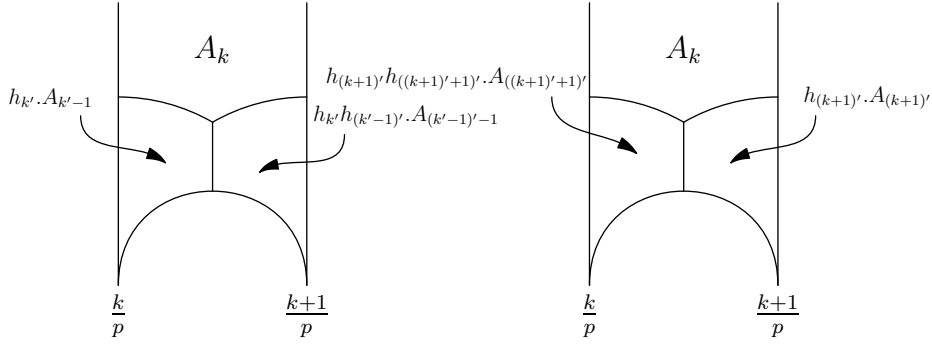}
\end{center}
\caption{Neighboring translates for $k\in\{1,\ldots, p-2\}$}
\label{glue4pp}
\end{figure}

From these we read off the presentation
\begin{equation}\label{presentation}
 \Gamma_p = \left\langle T, h_1, \ldots, h_{p-1} \left\vert\ 
\begin{split}
&h_{j'}h_j = \id\ \text{for $j=1,\ldots, p-1$,}\\
&h_{(k'-1)-1}h_{k'-1}h_k=\id\ \text{for $k=1,\ldots,p-2$}
\end{split}
\right.\right\rangle
\end{equation}
as well as the identities
\begin{equation}\label{identities}
 (k'-1)'-1 = (k+1)' \quad\text{and}\quad h_{k'}h_{(k'-1)'} = h_{(k+1)'}.
\end{equation}

\subsection{Cross section}\label{sec_cs} 
Let 
\[
 \pi\colon S\h \to SY_p = \Gamma_p\backslash S\h
\]
denote the canonical quotient map from the unit tangent bundle $S\h$ of $\h$ to that of $Y_p$. In the following we define a cross section for the geodesic flow on $Y_p$ via a set of representatives for it in $S\h$. For $k\in \{0,\ldots, p-1\}$ let
\[
C'_k \sceq \left\{ X\in S\h \left\vert\ X = a \frac{\partial}{\partial x}\vert_{\frac{k}{p}+iy} + b\frac{\partial}{\partial y}\vert_{\frac{k}{p}+iy},\ a>0,\ b\in\R,\ y>0\right.\right\},
\]
and
\[
C'_p \sceq \left\{ X\in S\h \left\vert\ X= a\frac{\partial}{\partial x}\vert_{iy} + b\frac{\partial}{\partial y}\vert_{iy},\ a<0,\ b\in\R,\ y>0\right.\right\}.
\]
The sets $C'_k$ ($k=0,\ldots, p-1$) consist of the unit tangent vectors based on the geodesic arc $\frac{k}{p}+i\R^+$ which point to the right. The set $C'_p$ consists of the unit tangent vectors based on $i\R^+$ which point to the left (cf.\@ Figure~\ref{crossreprpp}). 

\begin{figure}[h]
\begin{center}
\includegraphics{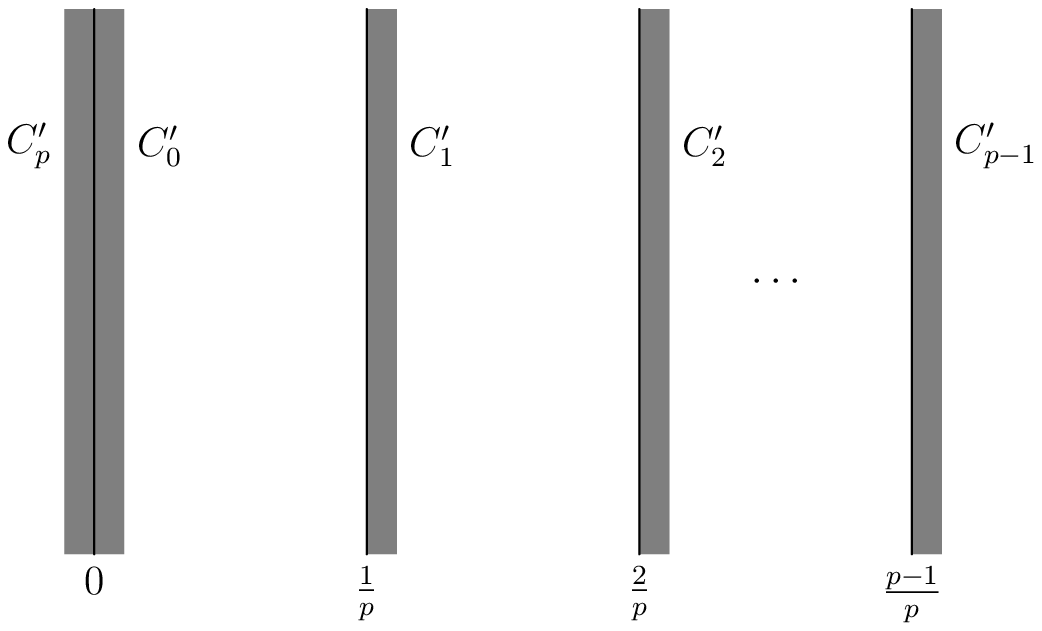}
\end{center}
\caption{}
\label{crossreprpp}
\end{figure}

Let
\[
 C' \sceq \bigcup_{k=0}^p C'_k
\]
and
\[
 \wh C \sceq \pi(C').
\]
In \cite{Hilgert_Pohl, Pohl_Symdyn2d} it is shown that $\wh C$ is a cross section for the geodesic flow on $Y_p$, and $C'$ is a set of representatives for it.

\textbf{Discrete dynamical system.}
Suppose that $\wh v$ is a vector in $\wh C$ and let $\wh \gamma_v$ denote the (unit speed) geodesic on $Y_p$ determined by 
\[
 \frac{d}{dt}\vert_{t=0}\wh\gamma_v(t) = \wh v.
\]
At least in the case that $\wh\gamma_v$ does not go into a cusp there is a minimal time $t(\wh v)>0$ such that 
\[
 \frac{d}{dt}\vert_{t=t(\wh v)}\wh\gamma_v(t) 
\]
is an element of $\wh C$. This gives rise to the partially defined first return map
\[
 \mc R\colon \wh C \to \wh C,\quad \wh v \mapsto  \frac{d}{dt}\vert_{t=t(\wh v)}\wh\gamma_v(t).
\]
The first return map $\mc R$ is conjugate to a discrete dynamical system on parts of the geodesic boundary $\R\cup\{\infty\}$ of $\h$ as seen in the following. For each vector $\wh v$ in $\wh C$ we have a unique representative $v$ in $C'$. If $v$ is contained in the component $C'_k$ for some $k=k(v) \in \{0,\ldots, p\}$, then we identify $\wh v$ with the triple $(\gamma_v(-\infty), \gamma_v(+\infty), k)$ where $\gamma_v$ is the geodesic on $\h$ determined by  
\[
 \frac{d}{dt}\vert_{t=0}\gamma_v(t) = v.
\]
Since $\gamma_v$ is a representative of $\wh\gamma_v$, the vector 
\[
 \mc R(\wh v) = \frac{d}{dt}\vert_{t=t(\wh v)}\wh\gamma_v(t)
\]
has a lift to $S\h$ which is tangent to $\gamma_v$. More precisely, it is represented by the first intersection of the set of tangent vectors to $\gamma_v(\R^+)$ with $\Gamma_p.C'$. Suppose that 
\[
 \frac{d}{dt}\vert_{t=t(\wh v)}\gamma_v(t) \in h.C'_\ell
\]
for some $h\in\Gamma_p$ and $\ell\in\{0,\ldots,p\}$. Then $\mc R(\wh v)$ is identified with 
\[
(h^{-1}.\gamma_v(-\infty), h^{-1}.\gamma_v(+\infty), \ell).
\]
Even though the geodesic $\gamma_v$ may intersect more than one of the components $C'_j$ ($j=0,\ldots,p$) of $C'$, it intersects each one in at most one vector. Therefore, the map 
\[
 \wh v \mapsto (\gamma_v(-\infty),\gamma_v(+\infty),k(v))
\]
is injective. In turn, the first return map $R$ is conjugate to the discrete dynamical system $\wt F$ defined on 
\[
 \left\{ (x,y,k) \left\vert\ 
\begin{split}
&x,y\in\R\cup\{\infty\},\ k\in\{0,\ldots,p\}, 
\\
&\exists\, \wh v\in\wh C\colon (\gamma_v(-\infty),\gamma_v(+\infty),k(v)) = (x,y,k)
\end{split}
\right.\right\}
\]
which maps $(\gamma_v(-\infty),\gamma_v(+\infty),k(v))$ to $(h^{-1}.\gamma_v(-\infty),h^{-1}.\gamma_v(+\infty),\ell)$ (in the notation from above). 

In the following, we are interested only in the forward part (the last two components) of $\wt F$. Let $F$ denote this restricted discrete dynamical system and let $D$ denote its domain of definition. To provide an explicit expression for $(D,F)$ we can deduce the relevant $\Gamma_p$-translates of $C'$ from the side pairing of the fundamental domain $\mc F$, see Figures~\ref{cross1pp} and \ref{cross2pp}.

\begin{figure}[h]
\begin{center}
\includegraphics{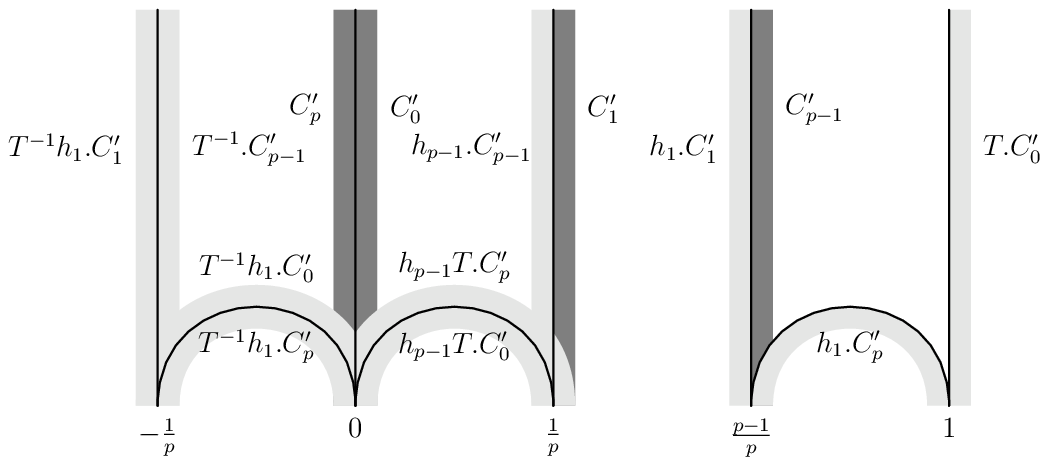}
\end{center}
\caption{}
\label{cross1pp}
\end{figure}

\begin{figure}[h]
\begin{center}
\includegraphics{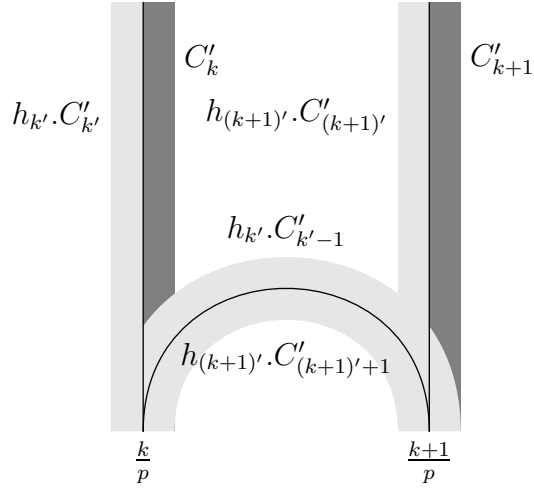}
\end{center}
\caption{Relevant $\Gamma_p$-translates for $k\in\{1,\ldots, p-2\}$}
\label{cross2pp}
\end{figure}

To simplify notation we set
\begin{align*}
 I_k &\sceq \left( \frac{k}{p},\infty\right) \qquad \text{for $k=0,\ldots, p-1$}
\intertext{and}
I_p & \sceq (-\infty,0).
\end{align*}
Then 
\[
  D = \bigcup_{k=0}^{p}  I_k \cup \{k\}.
\]
The action of $F$ on $D$ is given by the local diffeomorphisms
\begin{align*}
\left(-\infty, -\frac1p\right)\times \{p\} &\to I_1\times\{1\},\quad &(x,p) &\mapsto (h_{p-1}T.x,1)
\\
\left(-\frac1p,0\right)\times\{p\} &\to I_p\times\{p\},\quad &(x,p)&\mapsto (h_{p-1}T.x,p)
\\
\left(0,\frac1p\right)\times\{0\} &\to I_0\times\{0\},\quad &(x,0)&\mapsto (T^{-1}h_1.x,0)
\\
\left(\frac1p,\infty\right)\times\{0\} &\to I_1\times\{1\},\quad &(x,0)&\mapsto (x,1)
\\
\left(\frac{p-1}{p},1\right)\times \{p-1\} &\to I_p\times\{p\},\quad &(x,p-1)&\mapsto (h_{p-1}.x,p)
\\
(1,\infty)\times \{p-1\}&\to I_0\times\{0\},\quad &(x,p-1) &\mapsto (T^{-1}.x,0),
\end{align*}
and for $k\in\{1,\ldots, p-2\}$,
\begin{align*}
\left(\frac{k}{p},\frac{k+1}{p}\right)\times\{k\} &\to I_{(k+1)'+1}\times \{ (k+1)'+1\}, \quad &(x,k) &\mapsto (h_{k+1}.x, (k+1)'+1)
\\
\left(\frac{k+1}{p},\infty\right)\times \{k\} &\to I_{k+1}\times\{k+1\}, \quad &(x,k) &\mapsto (x,k+1).
\end{align*}

\subsection{Associated family of transfer operators}
For each $s\in\C$ the transfer operator $\mc L_{F,s}$ with parameter $s$ associated to $(D,F)$ is the operator 
\[
 \big(\mc L_{F,s}f\big)(x) \sceq \sum_{y\in F^{-1}(x)} \frac{f(y)}{|F'(y)|^s}
\]
defined on the space $\Fct(D;\C)$ of complex-valued functions on $D$. The structure of $F$ allows us to provide a matrix representation for $\mc L_{F,s}$. To that end we represent a function $f\in \Fct(D;\C)$ as
\[
 f =
\begin{pmatrix}
f\cdot 1_{I_0\times\{0\}}
\\
f\cdot 1_{I_1\times\{1\}}
\\
\vdots
\\
f\cdot 1_{I_{p-1}\times\{p-1\}}
\\
f\cdot 1_{I_p\times\{p\}}
\end{pmatrix}
\seqc
\begin{pmatrix}
 f_0
\\
f_1
\\
\vdots
\\
f_{p-1}
\\
f_p
\end{pmatrix}.
\]
Because of this vector structure, we may identify $I_k\times\{k\}$ with $I_k$ for any $k\in\{0,\ldots, p\}$.

Further, for $s\in\C$ and any function $\varphi\colon V\to\R$ on some subset $V$ of $\R$ we define the action
\begin{equation}\label{action}
 \big(\tau_s(g^{-1})\varphi\big)(t) \sceq \big(g'(t)\big)^s \varphi(g.t)
\end{equation}
of $g\in\Gamma_p$ on $\varphi$ whenever it is well-defined. Thus, if $g=\textbmat{a}{b}{c}{d}$, then 
\[
 \big(\tau_s(g^{-1})\varphi\big)(t) = \big( (ct+d)^{-2} \big)^s \varphi\left(\frac{at+b}{ct+d}\right).
\]
This is understood as a limit if $c\not=0$ and $t=-\frac{d}{c}$. 

The action of the transfer operator $\mc L_{F,s}$ on $f\in \Fct(D;\C)$ now becomes
\[
 \wt f = \mc L_{F,s} f
\]
with
\begin{align}
\label{eigen1} \wt f_0 & = \tau_s(T^{-1}h_1) f_0 + \tau_s(T^{-1})f_{p-1} &&\text{on $I_0$,}
\\
\label{eigen2} \wt f_p & = \tau_s(h_{p-1}) f_{p-1} + \tau_s(h_{p-1}T) f_p && \text{on $I_p$,}
\\
\label{eigen3} \wt f_1 & = f_0 + \tau_s(h_{p-1}T)f_p && \text{on $I_1$,}
\\
\label{eigen4} \wt f_{k+1} & = f_k + \tau_s(h_{k'}) f_{k'-1} && \text{on $I_{k+1}$}
\end{align}
for $k\in\{1,\ldots, p-2\}$.

We remark that the expression for $\mc L_{F,s}$ can directly be read off from Figures~\ref{cross1pp} and \ref{cross2pp}.

\subsection{Period functions}\label{sec_pf}
For $s\in \C$ let $\FE_s^{\omega,\dec}(\Gamma_p)$ be the space of function vectors 
\[
 f = 
\begin{pmatrix}
f_0 
\\
\vdots
\\
f_p
\end{pmatrix}
\]
such that 
\begin{enumerate}
\item[(PF1)] $f_j \in C^\omega(I_j;\C)$ for $j\in\{0,\ldots, p\}$,
\item[(PF2)] $f=\mc L_{F,s}f$,
\item[(PF3)] for $k\in\{1,\ldots,p-1\}$, the map
\[
\begin{cases}
f_k & \text{on $\left(\frac{k}{p},\infty\right)$}
\\
-\tau_s(h_{k'})f_{k'} & \text{on $\left(-\infty,\frac{k}{p}\right)$}
\end{cases}
\]
extends smoothly to $\R$, and
\item[(PF4)] the map
\[
\begin{cases}
-f_0 & \text{on $(0,\infty)$}
\\
f_p & \text{on $(-\infty,0)$}
\end{cases}
\]
extends smoothly to $P^1(\R)$.
\end{enumerate}

The notion of smooth extension to $P^1(\R)$ depends here on the parameter $s\in\C$. A function $\varphi\colon \R\to\C$ is said to extend smoothly to $P^1(\R)$ if and only if for some (and indeed any) element $g\in\Gamma_p$ which does not stabilize $\infty$, the functions $\varphi$ and $\tau_s(g)\varphi$ are smooth on $\R$.

\begin{remark}\label{smooth}
For $k\in\{1,\ldots,p-1\}$ let 
\[
 \varphi_k \sceq 
\begin{cases}
f_k & \text{on $\left(\frac{k}{p},\infty\right)$}
\\
-\tau_s(h_{k'})f_{k'} & \text{on $\left(-\infty,\frac{k}{p}\right)$}
\end{cases}
\]
be the map in (PF3). Then $\tau_s(h_k)\varphi_k = -\varphi_{k'}$. Thus, $\varphi_k$ extends smoothly to $P^1(\R)$.
\end{remark}

In Section~\ref{proofmainthm} below we will prove the following relation between Maass cusp forms and $1$-eigenfunctions of the transfer operator $\mc L_{F,s}$.

\begin{thm}\label{mainthm_coarse}
For $s\in\C$ with $0<\Rea s<1$, the space $\FE_s^{\omega,\dec}(\Gamma_p)$ is in linear isomorphism with the space of Maass cusp forms for $\Gamma_p$ with eigenvalue $s(1-s)$.
\end{thm}

Motivated by Theorem~\ref{mainthm_coarse}, we call the elements of $\FE_s^{\omega,\dec}(\Gamma_p)$ \textit{period functions} for the Maass cusp forms for $\Gamma_p$ with eigenvalue $s(1-s)$. 

% flatex input end: [discretization.tex]

% \usepackage[notref,notcite]{showkeys}
% flatex input: [periodfunctions.tex]
\section{Period functions and Maass cusp forms}\label{proofmainthm}

For the proof of Theorem~\ref{mainthm_coarse} we use the characterization of Maass cusp forms in parabolic $1$-cohomology, which was recently developed in \cite{BLZ_part2}. In Section~\ref{sec_charac} below we briefly recall this characterization.
Theorem~\ref{mainthm_coarse} follows then from showing that parabolic $1$-cocycle classes are in linear isomorphism with period functions, which is done in Section~\ref{proofthmfine} below. The isomorphism between period functions and parabolic $1$-cocycle classes is constructive as well as the isomorphism between parabolic $1$-cohomology and Maass cusp forms from \cite{BLZ_part2}.

We denote the space of Maass cusp forms for $\Gamma_p$ with eigenvalue $s(1-s)$ by $\MCF_s(\Gamma_p)$.

\subsection{Maass cusp forms as parabolic $1$-cocycle classes}\label{sec_charac}
Let $s\in\C$. Recall that the space $\mc V_s^{\omega*,\infty}$ of semi-analytic smooth vectors in the line model of the principal series representation with spectral parameter $s$ is the space of complex-valued functions $\varphi$ on $\R$ which are 
\begin{enumerate}[(a)]
\item\label{pc1} smooth  and extend smoothly to $P^1(\R)$,
\item and real-analytic on $\R\setminus E$ for a finite subset E which depends on $\varphi$.
\end{enumerate}
The action of $\Gamma_p$ on $\mc V_s^{\omega*,\infty}$ is given by the action $\tau_s$ in \eqref{action}. The notion of smooth extension to $P^1(\R)$ depends on the parameter $s$ in the same way as for period functions.

In the notation of restricted cocycles, the space of $1$-cocycles of group cohomology of $\Gamma_p$ is
\[
 Z^1(\Gamma_p;\mc V_s^{\omega*,\infty}) = \{ c\colon \Gamma_p \to \mc V_s^{\omega*,\infty} \mid \forall\, g,h\in\Gamma_p\colon c_{gh} = \tau_s(h^{-1})c_g + c_h \}.
\]
Here we write $c_g \in V_s^{\omega*,\infty}$ for the image $c(g)$ of $g\in\Gamma_p$ under the map $c\colon \Gamma_p \to \mc V_s^{\omega*,\infty}$. The space of parabolic $1$-cocycles is 
\[
 Z^1_\parab(\Gamma_p;\mc V_s^{\omega*,\infty}) = \left\{c\in Z^1(\Gamma_p;\mc V_s^{\omega*,\infty}) \left\vert\ 
\begin{split}
& \forall\, g\in\Gamma_p\ \text{parabolic}\ \exists\, \psi\in\mc V_s^{\omega*,\infty} \colon 
\\
& c_g = \tau_s(g^{-1})\psi -\psi 
\end{split}
\right.\right\}
\]
The spaces of the $1$-coboundaries of group cohomology and of parabolic cohomology are identical. They are given by
\[
 B^1_\parab(\Gamma_p;\mc V_s^{\omega*,\infty}) = B^1(\Gamma_p;\mc V_s^{\omega*,\infty}) = \{ g\mapsto \tau_s(g^{-1})\psi-\psi \mid \psi\in \mc V_s^{\omega*,\infty}\}.
\]
The parabolic $1$-cohomology space is the quotient space
\[
 H^1_\parab(\Gamma_p;\mc V_s^{\omega*,\infty}) = Z^1_\parab(\Gamma_p;\mc V_s^{\omega*,\infty})/B^1_\parab(\Gamma_p;\mc V_s^{\omega*,\infty}).
\]

\begin{thm}\cite{BLZ_part2}\label{parab_char}
For $s\in\C$ with $0<\Rea s<1$, the spaces $\MCF_s(\Gamma_p)$ and $H^1_\parab(\Gamma_p;\mc V_s^{\omega*,\infty})$ are isomorphic as vector spaces.
\end{thm}

The isomorphism in Theorem~\ref{parab_char} is given by the following integral transform: Let $R\colon \R\times\h \to\h$,
\[
 R(t,z) \sceq \Ima\left( \frac{1}{t-z} \right)
\]
denote the Poisson kernel, and let 
\[
 [u,v] \sceq \frac{\partial u}{\partial z}\cdot vdz + u\cdot\frac{\partial v}{\partial\overline z} d\overline z
\]
denote the Green form for two complex-valued smooth functions $u,v$ on $\h$. Let $u$ be a Maass cusp form for $\Gamma_p$ with eigenvalue $s(1-s)$ and suppose that $[c]$ is the parabolic $1$-cocycle class in $H^1_\parab(\Gamma_p;\mc V_s^{\omega*,\infty})$ which is associated to $u$. We pick any point $z_0\in\h$. Then $[c]$ is represented by the cocycle $c$ with 
\begin{equation}\label{def_cocycle}
 c_g(t) \sceq \int_{g^{-1}.z_0}^{z_0} [u,R(t,\cdot)^s]
\end{equation}
for $g\in\Gamma_p$. The integration is performed along any differentiable path in $\h$ from $g^{-1}.z_0$ to $z_0$, e.g.\@ the connecting geodesic arc. Since $[u,R(t,\cdot)^s]$ is a closed $1$-form (\cite{Lewis_Zagier}), the integral is well-defined. Varying the choice of $z_0$ changes $c$ by a parabolic $1$-coboundary. Recall the parabolic element $T=\textbmat{1}{1}{0}{1}$ from \eqref{T}. Then \cite[Proposition~4.5]{BLZ_part2} implies (cf.\@ \cite[Proposition~14.2]{BLZ_part2}) that each parabolic $1$-cocycle class $[c]$ has a unique representative $c$ with $c_T = 0$.

Since Maass cusp forms decay rapidly towards any cusp and $\infty$ is a cuspidal point, we may choose $z_0 = \infty$ in \eqref{def_cocycle} (cf.\@ the discussion in \cite{Lewis, Moeller_Pohl}). Then the path of integration has to be essentially contained in $\h$ (see Lemma~\ref{realana} below for are more precise statement). We remark that the choice $z_0=\infty$ yields the unique representative $c$ of the cocycle class $[c]$ for which $c_T = 0$.

The following lemma is now implied by standard theorems on parameter integrals.

\begin{lemma}\label{realana}
Let $u \in \MCF_s(\Gamma_p)$. If $a,b$ are two distinct cuspidal points with $a<b$, then the parameter integral
\[
 t\mapsto \int_a^b [u, R(t,\cdot)^s]
\]
defines a function which is smooth on $\R$ and real-analytic on  $(a,b) \cup (\R\setminus [a,b])$. Here the integration is performed along the geodesic from $a$ to $b$, or any differentiable path $\gamma\colon (c,d) \to \h$ with $\lim_{r\to c}\gamma(r) = a$, $\lim_{r\to d}\gamma(r) = b$ (in the cone topology of the geodesic compactification of $\h$), or any path which is piecewise of this form. In particular, for any $g\in\Gamma_p$, the function
\[
 c_g(t) = \int_{g^{-1}.\infty}^\infty [u, R(t,\cdot)^s]
\]
is smooth on $\R$ and real-analytic on the intervals $(-\infty, g^{-1}.\infty)$ and $(g^{-1}.\infty,\infty)$, and
\[
 t\mapsto \int_0^\infty [u,R(t,\cdot)^s]
\]
is smooth on $\R$ and real-analytic on $(-\infty,0)\cup (0,\infty)$.
\end{lemma}

Finally, as in \cite[Chap.\@ II.2]{Lewis_Zagier} it follows that 
\begin{equation}\label{transform}
 \tau_s(g^{-1}) \int_a^b [u,R(t,\cdot)^s] = \int_{g^{-1}.a}^{g^{-1}.b}[u,R(t,\cdot)^s]
\end{equation}
for any $g\in\Gamma_p$ and any cuspidal points $a,b$.

To end this section, we use Lemma~\ref{realana} and \eqref{transform} to improve the understanding of the fine structure of parabolic $1$-cocycles.

\begin{lemma}\label{finestructure}
Let $\Rea s\in (0,1)$ and let $c\in Z^1_\parab(\Gamma_p;\mc V_s^{\omega*,\infty})$ such that $c_T = 0$. Suppose that $u$ is the Maass cusp form associated to the cocycle class $[c]$. Then 
\begin{equation}\label{unique}
 c_{h_{p-1}T} = \tau_s(T^{-1}h_1)\psi - \psi
\end{equation}
with $\psi\in\mc V_s^{\omega*,\infty}$ determined by
\[
 \psi(t) = -\int_0^\infty [u,R(t,\cdot)^s],\quad t\in\R.
\]
Moreover, $\psi$ is the unique element in $\mc V_s^{\omega*,\infty}$ such that \eqref{unique} is satisfied.
\end{lemma}

\begin{proof}
The element 
\[
 h_{p-1}T = \bmat{1}{0}{p}{1}
\]
is parabolic. By definition there exists an element $\psi\in\mc V_s^{\omega*,\infty}$ such that 
\[
c_{h_{p-1}T} = \tau_s(T^{-1}h_1)\psi - \psi.
\] 
From \cite[Proposition~4.5]{BLZ_part2} it follows that $\psi$ is unique. By Lemma~\ref{realana} and \eqref{transform}, the map $\psi\colon \R \to \C$,
\[
 \psi(t) \sceq -\int_0^\infty [u,R(t,\cdot)^s]
\]
is an element of $\mc V_s^{\omega*,\infty}$. The normalization of $c$ yields that 
\[
 c_{h_{p-1}T}(t) = \int_{T^{-1}h_1.\infty}^\infty [u,R(t,\cdot)^s].
\]
Now a straightforward calculation shows that this function satisfies \eqref{unique}.
\end{proof}

\subsection{Period functions and parabolic cohomology}\label{proofthmfine}

Let $[c]\in H^1_{\parab}(\Gamma_p; \mc V_s^{\omega^*,\infty})$ be a parabolic $1$-cocycle class, and let $c\in Z^1_\parab(\Gamma_p; \mc V_s^{\omega^*,\infty})$ be its unique representative such that $c_T = 0$. We associate to $[c]$ a function vector 
\[
 f([c]) = \begin{pmatrix} f_0 \\ \vdots \\ f_p \end{pmatrix}
\]
as follows. For $k\in\{1,\ldots,p-1\}$ we define
\[
 f_k \sceq c_{h_k}\vert_{I_k}.
\]
Let $\psi$ be the unique (see Lemma~\ref{finestructure}) element in $\mc V_s^{\omega*,\infty}$ such that 
\[
 c_{h_{p-1}T} = \tau_s(T^{-1}h_1) \psi - \psi.
\]
We define
\[
 f_0 \sceq -\psi\vert_{I_0} \quad\text{and}\quad f_p\sceq \psi\vert_{I_p}.
\]
Conversely, given $f\in \FE_s^{\omega,\dec}(\Gamma_p)$ we define a map 
\[
 c = c(f) \colon \Gamma_p \to \mc V_s^{\omega^*,\infty}
\]
by $c_T \sceq 0$ and 
\begin{equation}\label{defi_chk}
 c_{h_k} \sceq 
\begin{cases}
 f_k & \text{on $\left(\frac{k}{p},\infty\right)$}
\\
-\tau_s(h_{k'})f_{k'} & \text{on $\left(-\infty,\frac{k}{p}\right)$}
\end{cases}
\end{equation}
for $k\in\{1,\ldots, p-1\}$.

We remark that \eqref{defi_chk} actually defines a smooth, semi-analytic function on $P^1(\R)$ by (PF3) and Remark~\ref{smooth}.

\begin{thm}\label{mainthm_fine}
For $\Rea s \in (0,1)$, the map
\[
 \FE_s^{\omega,\dec}(\Gamma_p) \to H^1_\parab(\Gamma_p;\mc V_s^{\omega^*,\infty}),\quad f \mapsto [c(f)]
\]
is a linear isomorphism between the vector spaces $\FE_s^{\omega,\dec}(\Gamma_p)$ and $H^1_\parab(\Gamma_p;\mc V_s^{\omega^*,\infty})$. Its inverse map is given by
\[
 H^1_\parab(\Gamma_p;\mc V_s^{\omega^*,\infty}) \to \FE_s^{\omega,\dec}(\Gamma_p),\quad [c] \mapsto f([c]).
\]
\end{thm}

By Theorem~\ref{parab_char}, Theorem~\ref{mainthm_fine} immediately yields Theorem~\ref{mainthm_coarse}. The proof of Theorem~\ref{mainthm_fine} is split into Propositions~\ref{cohom_period} and \ref{period_cohom} below.

\begin{prop}\label{cohom_period}
Let $s\in\C$ with $1>\Rea s>0$. If $[c]\in H^1_\parab(\Gamma_p;\mc V_s^{\omega^*,\infty})$, then $f([c])\in \FE_s^{\omega,\dec}(\Gamma_p)$.
\end{prop}

\begin{proof} 
Let $f\sceq f([c])$. We start by showing that $f$ is a $1$-eigenfunction of the transfer operator $\mc L_{F,s}$. Let $c\in Z^1_\parab(\Gamma_p; \mc V_s^{\omega^*,\infty})$ be the representative of $[c]$ with $c_T = 0$. By definition, for all $g\in\Gamma_p$ we have
\[
 c_g(t) = \int_{g^{-1}.\infty}^\infty [u,R(t,\cdot)^s]
\]
for a (unique) Maass cusp form $u$ with eigenvalue $s(1-s)$. Let $k\in\{1,\ldots, p-2\}$. On the interval $I_{k+1}$ it follows that
\begin{align*}
f_k + \tau_s(h_{k'})f_{k'-1} & = c_{h_k} + \tau_s(h_{k'})c_{h_{k'-1}}
\\
& = \int_{h_{k'}.\infty}^\infty [u,R(t,\cdot)^s] + \tau_s(h_{k'}) \int_{h_{(k'-1)'}.\infty}^\infty [u,R(t,\cdot)^s]
\\
& = \int_{h_{k'}.\infty}^\infty [u,R(t,\cdot)^s] + \int_{h_{k'}h_{(k'-1)'}.\infty}^{h_{k'}.\infty} [u,R(t,\cdot)^s]
\\
& = \int_{h_{k'}h_{(k'-1)'}.\infty}^\infty [u, R(t,\cdot)^s].
\end{align*}
Now $h_{k'}h_{(k'-1)'} = h_{(k+1)'}$ (see \eqref{identities}) yields
\begin{align*}
f_k + \tau_s(h_{k'})f_{k'-1} & = \int_{h_{(k+1)'}.\infty}^\infty [u,R(t,\cdot)^s] = c_{h_{k+1}} = f_{k+1}
\end{align*}
on $I_{k+1}$. Recall that $c_{h_{p-1}T} = \tau_s(T^{-1}h_1)\psi-\psi$ with 
\[
 \psi(t) = -\int_0^\infty [u,R(t,\cdot)^s].
\]
Then the remaining identifies follow easily by straightforward manipulations. Thus, $f$ is indeed a $1$-eigenfunction of $\mc L_{F,s}$. Lemma~\ref{realana} immediately shows that $f_k$ is real-analytic on $I_k$ for each $k\in\{0,\ldots, p\}$.

Since $c_{h_k} = -\tau_s(h_{k'})c_{h_{k'}}$ for all $k\in\{1,\ldots, p-1\}$, the map 
\[
\begin{cases}
f_k & \text{on $\left( \frac{k}{p},\infty\right)$}
\\
-\tau_s(h_{k'})f_{k'} & \text{on $\left(-\infty,\frac{k}{p}\right)$}
\end{cases}
\]
is a restriction of $c_{h_k}$ and hence extends smoothly to $\R$. Finally, the map 
\[
\begin{cases}
 -f_0 & \text{on $(0,\infty)$}
\\
f_p & \text{on $(-\infty,0)$}
\end{cases}
\]
is a restriction of $\psi$. In turn it extends smoothly to $P^1(\R)$.
\end{proof}

\begin{prop}\label{period_cohom}
Let $s\in\C$ with $\Rea s \in (0,1)$ and $f\in\FE_s^{\omega,\dec}(\Gamma_p)$. Then $c(f) \in Z^1_\parab(\Gamma_p;\mc V_s^{\omega^*,\infty})$.
\end{prop}

\begin{proof}
Let $c\sceq c(f)$. We start by proving that $c$ is well-defined. The presentation \eqref{presentation} of $\Gamma_p$ implies that it suffices to prove
\[
 c_{h_{j'}h_j} = 0  \quad\text{and}\quad c_{h_{(k'-1)-1}h_{k'-1}h_k}  = 0
\]
for $j\in\{1,\ldots, p-1\}$ and $k\in\{1,\ldots, p-2\}$. By the remark preceding Theorem~\ref{mainthm_fine}, it is enough to establish these identities on a dense subset of $P^1(\R)$. For $j\in\{1,\ldots, p-1\}$ property (PF3) for $f$ immediately yields
\[
 c_{h_{j'}h_j} = \tau_s(h_{j'})c_{h_{j'}} + c_{h_j} = 0
\]
on $P^1(\R)$. Let $k\in\{1,\ldots, p-2\}$. Then 
\[
 c_{h_{(k'-1)-1}h_{k'-1}h_k} = \tau_s(h_{k}^{-1}h_{k'-1}^{-1}) c_{h_{(k'-1)'-1}} + \tau_s(h_{k'})c_{h_{k'-1}} + c_{h_k}.
\]
Since $h_k^{-1}h_{k'-1}^{-1} = h_{(k+1)'}$ and $(k'-1)'-1 = (k+1)'$ (see \eqref{identities}), it follows that
\begin{equation}\label{mustvan}
 c_{h_{(k'-1)-1}h_{k'-1}h_k} = \tau_s(h_{(k+1)'}) c_{h_{(k+1)'}} + \tau_s(h_{k'})c_{h_{k'-1}} + c_{h_k}.
\end{equation}
By definition we have
\[
c_{h_k} =
\begin{cases}
f_k & \text{on $\left(\frac{k}{p},\infty\right)$}
\\
-\tau_s(h_{k'})f_{k'} & \text{on $\left(-\infty,\frac{k}{p}\right)$.}
\end{cases}
\]
Moreover, 
\[
\tau_s(h_{k'}) c_{h_{k'-1}} =
\begin{cases}
\tau_s(h_{k'})f_{k'-1} & \text{on $\left(-\infty,\frac{k}{p}\right)\cup \left(\frac{k+1}{p},\infty\right)$}
\\
-\tau_s(h_{(k+1)'})f_{(k'-1)'} & \text{on $\left(\frac{k}{p},\frac{k+1}{p}\right)$} 
\end{cases}
\]
and
\[
\tau_s(h_{(k+1)'})c_{h_{(k+1)'}} = 
\begin{cases}
\tau_s(h_{(k+1)'})f_{(k+1)'} & \text{on $\left(-\infty,\frac{k+1}{p}\right)$}
\\
-f_{k+1} & \text{on $\left(\frac{k+1}{p},\infty\right)$.} 
\end{cases}
\]
Thus, \eqref{mustvan} becomes
\begin{equation}\label{fform}
c_{h_{(k'-1)-1}h_{k'-1}h_k} =
\begin{cases}
 -f_{k+1} + \tau_s(h_{k'})f_{k'-1} + f_k & \text{on $\left(\frac{k+1}{p},\infty\right)$}
\\
\tau_s(h_{(k+1)'})f_{(k+1)'} + f_k - \tau_s(h_{(k+1)'})f_{(k'-1)'} & \text{on $\left(\frac{k}{p},\frac{k+1}{p}\right)$}
\\
\tau_s(h_{(k+1)'})f_{(k+1)'} + \tau_s(h_{k'})f_{k'-1} - \tau_s(h_{k'})f_{k'} & \text{on $\left(-\infty,\frac{k}{p}\right)$.}
\end{cases}
\end{equation}
Since $f$ is a $1$-eigenfunction of $\mc L_{F,s}$ we have 
\begin{align}
\label{branch1} f_{k+1} & = f_k + \tau_s(h_{k'})f_{k'-1} && \text{on $\left(\frac{k+1}{p},\infty\right)$}
\\
\label{branch2} f_{(k'-1)'} & = f_{(k+1)'} + \tau_s(h_{k+1})f_k && \text{on $\left(\frac{(k'-1)'}{p},\infty\right)$}
\\
\label{branch3} f_{k'} & = f_{k'-1} + \tau_s(h_{(k'-1)'})f_{(k+1)'} && \text{on $\left(\frac{k'}{p},\infty\right)$.}
\end{align}
Here we used $(k'-1)'-1=(k+1)'$. Now \eqref{branch1} shows directly that the first branch of \eqref{fform} vanishes. Acting with $\tau_s(h_{(k+1)'})$ on \eqref{branch2} shows that the second branch of \eqref{fform} vanishes. Finally, acting with $\tau_s(h_{k'})$ on \eqref{branch3} shows that the third branch of \eqref{fform} vanishes.  

To complete the proof it remains to show that $c$ is parabolic. To that end we define
\[
 \psi \sceq 
\begin{cases}
-f_0 & \text{on $(0,\infty)$}
\\
f_p & \text{on $(-\infty,0)$.}
\end{cases}
\]
By (PF4) this defines an element of $\mc V_s^{\omega^*,\infty}$. We claim that
\[
 c_{h_{p-1}T} = \tau_s(T^{-1}h_1)\psi -\psi.
\]
Again it suffices to establish this identity on a dense subset of $P^1(\R)$. Since $c_T=0$, we have
\[
 c_{h_{p-1}T} = \tau_s(T^{-1})c_{h_{p-1}} + c_T = \tau_s(T^{-1})c_{h_{p-1}}
\]
with
\[
\tau_s(T^{-1})c_{h_{p-1}} = 
\begin{cases}
\tau_s(T^{-1})f_{p-1} & \text{on $\left(-\frac{1}{p},\infty\right)$}
\\
-\tau_s(T^{-1}h_1)f_1 & \text{on $\left(-\infty,-\frac{1}{p}\right)$.}
\end{cases}
\]
On $(0,\infty)$ we have
\begin{align*}
c_{h_{p-1}T} = \tau_s(T^{-1})c_{h_{p-1}} = f_0 -\tau_s(T^{-1}h_1)f_0 = \tau_s(T^{-1}h_1)\psi -\psi
\end{align*}
by \eqref{eigen1}. Here, \eqref{eigen2} is equivalent to 
\[
 \tau_s(T^{-1}h_1)f_p = \tau_s(T^{-1})f_{p-1} + f_p
\]
on $T^{-1}h_1.I_p = \left(-\frac{1}{p},0\right)$. Thus, on this interval we have
\[
c_{h_{p-1}T} = \tau_s(T^{-1}h_1)f_p - f_p = \tau_s(T^{-1}h_1)\psi - \psi.
\]
Finally, \eqref{eigen3} is equivalent to 
\[
 \tau_s(T^{-1}h_1)f_1 = \tau_s(T^{-1}h_1)f_0 + f_p
\]
on $T^{-1}h_1.I_1 = \left(-\infty,-\frac1p\right)$. This yields
\[
c_{h_{p-1}T} = -\tau_s(T^{-1}h_1)f_0 - f_p = \tau_s(T^{-1}h_1)\psi-\psi
\]
on $\left(-\infty,-\frac1p\right)$. Thus, the proof is complete.
\end{proof}

% flatex input end: [periodfunctions.tex]

% \usepackage[notref,notcite]{showkeys}
% flatex input: [sample.tex]
\section{Different choices of sets of representatives}\label{sec_examp}

The definition of period functions in Section~\ref{sec_pf} involved the choice of a set of representatives for the cross section. A different choice would lead to a different definition of period functions, for which an analog of Theorem~\ref{mainthm_fine} could be established. Thus, any two definitions of period functions arising in this way indeed define isomorphic sets of functions. In this section we present one example to illustrate the effect of a different choice of sets of representatives for the Hecke congruence subgroup $\Gamma_3$. Throughout we use 
\[
 h_1 = \bmat{2}{-1}{3}{-1} \quad\text{and}\quad h_2 = \bmat{1}{-1}{3}{-2}
\]
and $T=\textbmat{1}{1}{0}{1}$. We assume throughout that $\Rea s \in (0,1)$.

\textbf{The original choice.} With the set of representatives from Section~\ref{sec_cs} for the cross section we get the transfer operator
\[
 \mc L_s=
\begin{pmatrix}
\tau_s(T^{-1}h_1) & 0 & \tau_s(T^{-1}) & 0 
\\
1 & 0 & 0 & \tau_s(h_2T)
\\
0 & 1+\tau_s(h_2) & 0 & 0
\\
0 & 0 & \tau_s(h_2) & \tau_s(h_2T) 
\end{pmatrix}
\]
acting on function vectors
\[
 f =
\begin{pmatrix}
 f_0
\\
f_1
\\
f_2
\\
f_3
\end{pmatrix}
\]
where $f_0 \in \Fct((0,\infty);\C)$, $f_1\in\Fct((\frac13,\infty);\C)$, $f_2\in\Fct((\frac23,\infty);\C)$ and $f_3\in\Fct((-\infty,0);\C)$. The definition of period functions is then exactly the one from Section~\ref{sec_pf}.

\textbf{A different choice.} Now we choose the set of representatives 
\[
 \wt C' \sceq \bigcup_{k=0}^3 \wt C'_k
\]
as indicated in Figure~\ref{crossother}. 

\begin{figure}[h]
\begin{center}
\includegraphics{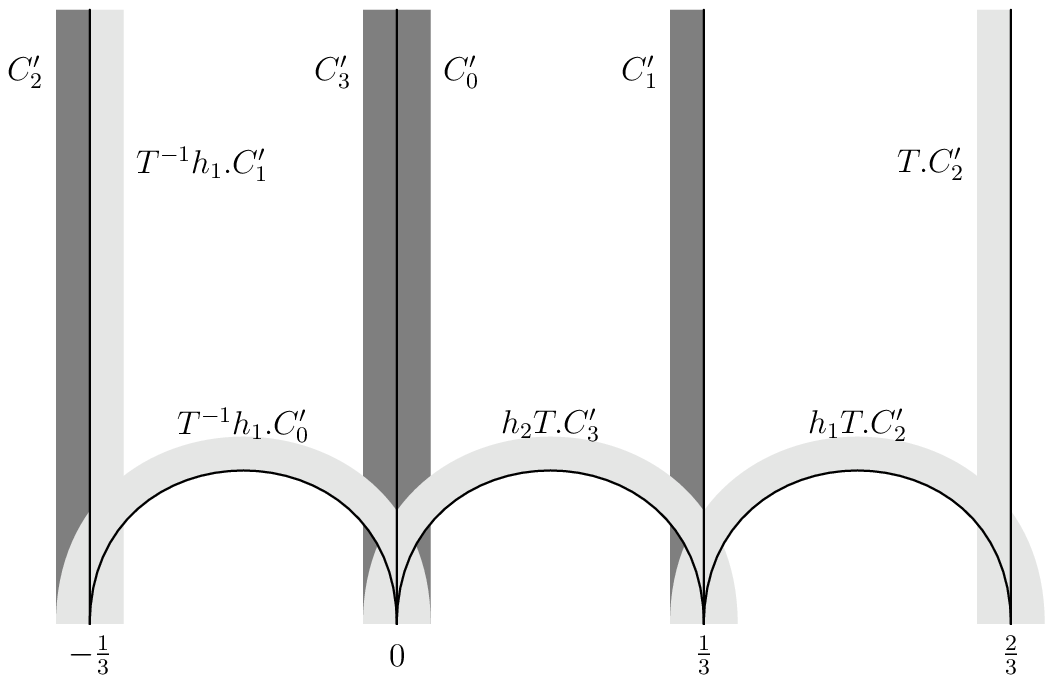}
\end{center}
\caption{}
\label{crossother}
\end{figure}

From this figure we can read off that the associated transfer operator is
\[
 \wt{\mc L}_s =
\begin{pmatrix}
\tau_s(T^{-1}h_1) & \tau_s(T^{-1}h_1) & 0 & 0 
\\
0 & 0 & \tau_s(T) + \tau_s(h_1T) & 0
\\
\tau_s(T^{-1}h_1) & 0 & 0 & 1
\\
0 & 1 & 0 & \tau_s(h_2T) 
\end{pmatrix}
\]
acting on function vectors
\[
 \wt f = 
\begin{pmatrix}
 \wt f_0
\\
\wt f_1
\\
\wt f_2
\\
\wt f_3
\end{pmatrix}
\]
where $\wt f_0\in \Fct( (0,\infty);\C)$, $\wt f_1\in\Fct((-\infty,\frac13);\C)$, $\wt f_2\in\Fct((-\infty,-\frac13);\C)$ and $\wt f_3\in\Fct((-\infty,0);\C)$. In this setting period functions are defined to be those function vectors 
\begin{itemize}
\item for which $\wt f_0,\ldots, \wt f_3$ are real-analytic on their respective domain of definition,
\item which are $1$-eigenfunctions of $\wt{\mc L}_s$,
\item for which the two maps
\[
\begin{cases}
\wt f_1 & \text{on $\left(-\infty,\frac13\right)$}
\\
-\tau_s(h_2T)\wt f_2 & \text{on $\left(\frac13,\infty\right)$}
\end{cases}
\]
and
\[
\begin{cases}
\wt f_2 & \text{on $\left(-\infty,-\frac13\right)$}
\\
-\tau_s(T^{-1}h_1)\wt f_1 & \text{on $\left(-\frac13,\infty\right)$}
\end{cases}
\]
extend smoothly to $\R$, and
\item for which the map
\[
\begin{cases}
-\wt f_0 & \text{on $(0,\infty)$}
\\
\wt f_1 & \text{on $(-\infty,0)$} 
\end{cases}
\]
extends smoothly to $P^1(\R)$.
\end{itemize}

\textbf{Relation between the two definitions of period functions.} Let $\FE_{s,1}^{\omega,\dec}(\Gamma_3)$ denote the space of period functions from the first definition, and $\FE_{s,2}^{\omega,\dec}(\Gamma_3)$ the space of period functions from the second definition. One easily proves the following proposition.

\begin{prop}
The spaces $\FE_{s,1}^{\omega,\dec}(\Gamma_3)$ and $\FE_{s,2}^{\omega,\dec}(\Gamma_3)$ are linear isomorphic. Such an isomorphism is provided by the map $\FE_{s,1}^{\omega,\dec}(\Gamma_3) \to \FE_{s,2}^{\omega,\dec}(\Gamma_3)$,
\[
 f = 
\begin{pmatrix}
f_0
\\
f_1
\\
f_2
\\
f_3
\end{pmatrix}
\mapsto 
\wt f =
\begin{pmatrix}
f_0
\\
\tau_s(h_2)f_2
\\
\tau_s(T^{-1}h_1)f_1
\\
f_3
\end{pmatrix}.
\]
\end{prop}

% flatex input end: [sample.tex]

% \usepackage[notref,notcite]{showkeys}
% \input{examples}

%FLATEX-REM:\bibliography{/home/anke/DatenDienst/latex/ap_bib}
%*flatex input: [Pohl_mcf_gamma0p.bbl]
\providecommand{\bysame}{\leavevmode\hbox to3em{\hrulefill}\thinspace}
\providecommand{\MR}{\relax\ifhmode\unskip\space\fi MR }
\providecommand{\MRhref}[2]{%
  \href{http://www.ams.org/mathscinet-getitem?mr=#1}{#2}
}
\providecommand{\href}[2]{#2}

% flatex input end: [Pohl_mcf_gamma0p.bbl]
%FLATEX-REM:\bibliographystyle{amsalpha}
\end{document}